\documentclass[twoside,11pt]{article}

\usepackage{amsmath}
\usepackage{amssymb}
\usepackage{amsthm}
\usepackage{amsfonts}
\usepackage{fancyhdr}
\usepackage{indentfirst}
\usepackage{appendix}
\usepackage{graphicx}
\usepackage{setspace}
\usepackage{cite,soul,ulem}
\usepackage{xcolor}
\usepackage{titlesec}
\usepackage{geometry}
\usepackage{hyperref}
\hypersetup{bookmarksnumbered=true,colorlinks=true,
linkcolor=blue,citecolor=blue,urlcolor=blue}

\makeatletter
\renewcommand\Hy@numberline[1]{#1. }
\makeatother

\titleformat{\section}[hang]
{\large\bfseries}{\thesection.}{0.48em}{}
\titleformat{\subsection}[runin]
{\normalsize\bfseries}{\thesubsection.}{0.48em}{}[.\MC]

\newtheorem{lemma}{Lemma}[section]

\newtheorem{theorem}{Theorem}[section]

\theoremstyle{definition}

\newtheorem{remark}{Remark}[section]

\numberwithin{equation}{section}
\allowdisplaybreaks[4]
\date{}

\def \p {\partial}
\def \l {\lambda}
\def \g {\gamma}
\def \d {\cdot}
\def \n {\nabla}
\def \la {\Delta}
\def \ze {\xi}
\def \vt {\vartheta}
\def \wi {\widetilde}
\def \ha {\widehat}
\def \ca {\mathcal}
\def \bb {\mathbb}

\newcommand{\MC}{\!}

\newcommand{\ie}{i.e.}
\newcommand{\eg}{e.g.}
\newcommand{\et}{\textit{et al.}}

\newcommand{\norm}[1]{\left\lVert#1\right\rVert}
\newcommand{\abs}[1]{\left|#1\right|}
\newcommand{\xkh}[1]{\left(#1\right)}
\newcommand{\zkh}[1]{\left[#1\right]}
\newcommand{\dkh}[1]{\left\{#1\right\}}

\newcommand{\lan}[1]{\left\langle#1\right\rangle}
\renewcommand{\Re}{\operatorname{Re}}

\begin{document}

\title{Linear stability analysis of the Couette flow for the 2D Euler-Poisson system
\footnotetext{$^*$Corresponding author.}
\footnotetext{\hspace{4pt}\textit{E-mail addresses}:
\href{mailto:puxueke@gmail.com}{puxueke@gmail.com}(X. Pu),
\href{mailto:wywlzhou@163.com}{wywlzhou@163.com}(W. Zhou),
\href{mailto:biandongfen@bit.edu.cn}{biandongfen@bit.edu.cn}(D. Bian)}}

\author{Xueke Pu$^1$, \ \ \ \ \ \ Wenli Zhou$^1$ \ \ \ \ and \ \ \ \  Dongfen Bian$^{2,*}$\\[1.26ex]
\textit{\normalsize $^1$School of Mathematics and Information Science, Guangzhou University,}\\[0.32ex]
\textit{\normalsize Guangzhou 510006, PR China}\\[0.32ex]
\textit{\normalsize $^2$School of Mathematics and Statistics, Beijing Institute of Technology,}\\[0.32ex]
\textit{\normalsize Beijing 100081, PR China}}

\maketitle\thispagestyle{empty}\vspace{-8mm}

\begin{quote}\small
\textbf{Abstract} This paper is concerned with the linear stability analysis for the Couette flow of the Euler-Poisson system for both ionic fluid and electronic fluid in the domain $\bb{T}\times\bb{R}$. We establish the upper and lower bounds of the linearized solutions of the Euler-Poisson system near Couette flow. In particular, the inviscid damping for the solenoidal component of the velocity is obtained.




\end{quote}

\section{Introduction}\label{se:1}

Consider the following two fluid Euler-Poisson system in a two dimensional domain $\bb{T}\times\bb{R}$,
\begin{equation}\label{2fluid}
\begin{split}
& \partial_t{\widetilde\eta_{\pm}} +\nabla\cdot (\widetilde\eta_{\pm}\widetilde v_{\pm})=0,\\
& \widetilde\eta_{\pm}m_{\pm}[\partial_t\widetilde v_{\pm}+\widetilde v_{\pm}\cdot\nabla \widetilde v_{\pm}] +T_{\pm}\nabla \widetilde p_{\pm} =\mp e \widetilde\eta_{\pm}\nabla\widetilde\phi,\\
& -\Delta\widetilde{\phi}=4\pi e [\widetilde\eta_+-\widetilde\eta_-],
\end{split}
\end{equation}
where ${\wi\eta_{\pm}}(x,y,t)$ and ${\wi v_{\pm}}(x,y,t)$ denote the density and the velocity of the ion fluid $(+)$ and the electron fluid $(-)$ in a plasma, $\n\wi{\phi}$ represents the self-consistent electric field, $\nabla {\widetilde p}_{\pm}$ denote the pressure of ions and electrons, $m_{\pm}$ denote the mass density of the ions and the electrons, $e$ is the charge for an electron and $T_{\pm}$ denote the effective temperatures of the ions and the electrons. Such a system is the simplified model from the two fluid Euler-Maxwell system as the light speed tends to infinity and is very fundamental in plasma physics. This Euler-Poisson system is also an important origin of many famous dispersive equations such as KdV, KP, NLS equations and is widely studied in the literature in recent years. See a series of papers of the first author and his collaborators \cite{Pu13,LP23, GP14,LP19,LBP24} for the rigorous mathematical justifications of these dispersive equations from the Euler-Poisson system. For more physics background and the mathematical development of such a system, one may see the concise review paper \cite{Guo2019} and the references cited therein.

The two-fluid Euler-Poisson system has rich and complex dynamics with distinct physical parameters. In particular, since the mass ratio of electrons and ions is of $10^{-3}\ll 1$, the system \eqref{2fluid} can be simplified to the following one-fluid Euler-Poisson system for the electron fluid dynamics of $({\wi\eta}_-,{\wi v}_-,\wi\phi)$ in an ion background (Langmuir waves, \cite{Guo1998,Guo2017})
\begin{flalign}
  &\p_t \wi{\eta}_-+\n \d (\wi{\eta}_-\wi{v}_-)=0,\label{eepfl:wie}\\
  &\wi{\eta}_-\zkh{\p_t \wi{v}_-+(\wi{v}_-\d\n)\wi{v}_-}
  +\frac{1}{m_-} \n {\wi p}_{-}(\wi{\eta}_-)=\frac{e}{m_-}\wi{\eta}_-\n\wi{\phi},\label{eepfl:npw}\\
  &\la\wi{\phi}=4 \pi e\wi{\eta}_{-}-4 \pi e.\label{eepfl:e0w}
\end{flalign}

On the other hand, assume $\widetilde p_{-}(\widetilde\eta_{-})=T_-\widetilde\eta_-$ in \eqref{2fluid} with a constant temperature $T_-$ and take formally $m_-/m_+\to0$ (such a limit is rigorously justified by Grenier \et~recently in \cite{GGPS19}), we deduce $\nabla\widetilde p_-=T_-\nabla\widetilde\eta_-=e\widetilde\eta_-\nabla\tilde\phi$, from which the celebrated Boltzmann relation for the electron density can be derived:
$$\widetilde\eta_-=\eta_0\exp{\{e\widetilde\phi/T_-\}},$$
where $\eta_0$ is a constant, set to be one for simplicity. Then we obtain the reduced Euler-Poisson system for ion dynamics for $(\widetilde\eta_+,\widetilde v_+,\widetilde\phi)$
\begin{flalign}
  &\p_t \wi{\eta}_{+}+\n\d (\wi{\eta}_{+}\wi{v}_{+})=0,\label{fl:wie}\\
  &\wi{\eta}_{+}\zkh{\p_t \wi{v}_{+}+(\wi{v}_{+}\d\n)\wi{v}_{+}}
  +\frac{T_+}{m_+} \n\wi{\eta}_{+}=-\frac{e}{m_+}\wi{\eta}_{+}\n\wi{\phi},\label{fl:npw}\\
  &\la\wi{\phi}=4 \pi e\zkh{\exp\xkh{\frac{e}{T_-}\wi{\phi}}-\wi{\eta}_{+}}.\label{fl:e0w}
\end{flalign}
For global well-posedness of such ion dynamics system in $\bb{R}^{3+1}$, one may refer to \cite{GP11}.

 In this paper, we are interested in another topic of this system, i.e., stability of the Couette flow of the above Euler-Poisson system. To be precise, we will consider the linear stability of the Couette flow for the two Euler-Poisson systems \eqref{eepfl:wie}-\eqref{eepfl:e0w} and \eqref{fl:wie}-\eqref{fl:e0w} in a periodic strip $\bb{T}\times\bb{R}$ with $\bb{T}=\bb{R}/\bb{Z}$. It is clear that there is a stationary solution $$(\wi{\eta}_s, \wi{v}_s, \wi{\phi}_s)=({1, (y,0)^{\top}, 0})$$ for both system (\ref{fl:wie})-(\ref{fl:e0w})  and \eqref{eepfl:wie}-\eqref{eepfl:e0w}. We also call this solution Couette flow, following the name in Navier-Stokes equation in fluid dynamics. Formal stability analysis of plasma equilibria to the related Poisson-Vlasov and Maxwell-Vlasov systems in one, two and three dimensions is referred to the paper of Holm {\it et al} \cite{HM1985}. To the best of our knowledge, no rigorous mathematical stability analysis of the Couette flow for the Euler-Poisson system either for the one-fluid Euler-Poisson system or the two-fluid Euler-Poisson system can be found in the literature and we believe that this topic is new and interesting for the Euler-Poisson system. The purpose of the present paper is to study the linear stability of the Couette flow for both the ion and electron one-fluid Euler-Poisson dynamics.

Let us deviate slightly to briefly introduce some related results on the stability analysis of the Couette flow in fluid dynamics, in particular for the Navier-Stokes or the Euler system.  Indeed, from the mathematical point of view, the stability of the Couette flow and related topics have been extensively studied and many fruitful results were obtained in the past decades.  For an incompressible fluid, the classical results on the linear stability analysis of the Couette flow were obtained by Rayleigh\cite{lr1879} and Kelvin\cite{lk1887}. Inviscid damping for the Euler equations on the domain $\bb{T}\times\bb{R}$ was studied by Bedrossian and Masmoudi\cite{jb2015}, which is a significant result on nonlinear stability of the planar Couette flow. For nonlinear inviscid damping for a class of monotone shear flows for the 2D Euler equation in finite channel for initial perturbation in Gevrey class with compact support, one may refer to \cite{Masmoudi-zhao}. See also \cite{Ionescu-Jia-CMP-2020} for the inviscid damping near the Couette flow in a channel for the 2D Euler equation, where the authors showed that the velocity field converges to a nearby shear flow when the initial perturbation is small in a suitable Gevrey sapce.  One may also refer to other nice work \cite{DM2023,Zill2017,IJ2020,Jia2020ARMA,Jia2020SIMA,LZ2011ARMA,MZ2022,WZZ2016} for the stability analysis of the Couette flow or general shear flows for the Euler equation in various situations.  
In the presence of viscosity, some interesting results have been obtained in the study of transition threshold problem related to enhanced dissipation (see, \eg, \cite{BP22,jb2017,pg2020,nm2022,fw2018,dw2021,qc2020,MZ2022} for various situations). More results on the stability of the two-dimensional and three-dimensional Couette flow in the Navier-Stokes equations at high Reynolds number can be found in the review paper\cite{nm2019}.

In the case of compressible fluid, the linear stability of plane Couette flow has been studied in the past decades in physics literature in the past nearly a century starting from the '40s \cite{LL1946,HM1985,DR2004}.  Glatzel \cite{wg1988} considered the inviscid stability of this problem based on a normal mode analysis, and the effects of viscosity were then taken into account by the same author \cite{wg1989}. Duck \et\cite{pw1994} investigated the compressible stability of the plane Couette flow for realistic compressible flow models, where they show that the details of the mean flow profile have a profound effect on the stability of the flow. Hu and Zhong\cite{sx1998} studied the linear stability of viscous supersonic plane Couette flow for a perfect gas governed by Sutherland viscosity law by using a fourth-order finite-difference method and a spectral collocation method. For rigorous mathematical analysis, when the initial value is sufficiently close to the plane Couette flow, Kagei\cite{yk2011} proved that the plane Couette flow is asymptotically stable provided that the Reynolds and Mach numbers are sufficiently small, see also Kagei\cite{ya2011,yk2012} for more general parallel flows. Subsequently, the constraint on the Reynolds number in \cite{yk2011}, which is used to ensure the stability of the plane Couette flow, was relaxed by Li and Zhang\cite{hx2017}. The inviscid damping and enhanced dissipation phenomena were detected by Antonelli \et\cite{pm2021,ADM2020} for the homogeneous Couette flow in a 2D isentropic compressible fluid. In addition, we refer to \cite{xz2023} for a inviscid damping result on the non-isentropic compressible Euler equations on the domain $\bb{T}\times\bb{R}$. For the 3D case, Zeng \et\cite{lz2022} established the linear stability result of the Couette flow in the isentropic compressible Navier-Stokes equations on the domain $\bb{T}\times\bb{R}\times\bb{T}$.  However, the rigorous mathematical study is less developed compared to the incompressible case,  and a lot of work is to be established in both the linear and the nonlinear level in the future. In particular, there is no stability analysis for the Couette flow for the Euler-Poisson system studied in this paper.

In the following, we will discuss the linear stability of the Couette flow for the ion Euler-Poisson system and the electron Euler-Poisson system separately. Compared to the classical Euler system and the Navier-Stokes system, the analysis includes some new insights, since the system is nonlocal due to the presence of the electric field.

\subsection{Ion dynamics case}
To study the linear stability for the Couette flow in system (\ref{fl:wie})-(\ref{fl:e0w}), we consider a perturbation around the above stationary solution, \ie, let
\begin{equation*}
  \wi{\eta}_{+}=\eta_{+}+\wi{\eta}_s, \quad \wi{v}_{+}=u_{+}+\wi{v}_s, \quad \wi{\phi}=\phi+\wi{\phi}_s.
\end{equation*}
We first consider the ion dynamics system \eqref{fl:wie}-\eqref{fl:e0w}. Then the linearized system around the stationary solution $(\wi{\eta}_s, \wi{v}_s, \wi{\phi}_s)$ for system (\ref{fl:wie})-(\ref{fl:e0w}) is written as
\begin{flalign}
  &\p_t \eta_{+}+y\p_x \eta_{+}+(\n \d u_{+})=0,\label{fl:ndu}\\
  &\p_t u_{+}+y\p_x u_{+}+\begin{pmatrix}u^y_{+}\\0\end{pmatrix}+\frac{T_+}{m_+} \n \eta_{+}=-\frac{e}{m_+}\n\phi,\label{fl:ptu}\\
  &\la \phi=4 \pi e \xkh{\frac{e}{T_-}\phi-\eta_{+}}.\label{fl:phi}
\end{flalign}
Let $\psi_{+}=\n \d u_{+}$ and $\omega_{+}=\n^\perp \d u_{+}$ with $\n^\perp=(-\p_y,\p_x)^{\top}$. Then the Helmholtz projection operators are defined as
\begin{equation}
  u_{+}=(u^x_{+},u^y_{+})^{\top}=\n\la^{-1}\psi_{+}+\n^\perp\la^{-1}\omega_{+}:=\bb{Q}[u_{+}]+\bb{P}[u_{+}].\label{eq:pqu}
\end{equation}
Then the linearized system (\ref{fl:ndu})-(\ref{fl:phi}) is transformed into the form with respect to $(\eta_{+},\psi_{+},\omega_{+})$
\begin{flalign}
  &\p_t \eta_{+}+y\p_x \eta_{+}+\psi_{+}=0,\label{fl:eta}\\
  &\p_t \psi_{+}+y\p_x \psi_{+}+2\p_x u^y_{+}+\frac{T_+}{m_+} \la \eta_{+}
  =-\frac{4\pi e^2}{m_+}\la\xkh{-\la+\frac{4\pi e^2}{T_-}}^{-1}\eta_{+},\label{fl:psi}\\
  &\p_t \omega_{+}+y\p_x\omega_{+}-\psi_{+}=0.\label{fl:ome}
\end{flalign}

Denote the average of a function in the $x$ direction by
\begin{equation*}
  f_a(y)=\frac{1}{2\pi}\int_{\bb{T}}f(x,y)dx.
\end{equation*}We take the $x$-average for equations (\ref{fl:eta}), (\ref{fl:psi}) and (\ref{fl:ome}) to obtain the following system
\begin{flalign}
  &\p_t \eta_{+,a}+\psi_{+,a}=0,\label{fl:aa1}\\
  &\p_t \psi_{+,a}+\frac{T_+}{m_+} \p_{yy} \eta_{+,a}=-\frac{4\pi e^2}{m_+}\la\xkh{-\la+\frac{4\pi e^2}{T_-}}^{-1}\eta_{+,a},\label{fl:aa2}\\
  &\p_t \omega_{+,a}-\psi_{+,a}=0.\label{fl:aa3}
\end{flalign}Combining (\ref{fl:aa1}) with (\ref{fl:aa3}) yields $\p_t\xkh{\eta_{+,a}+\omega_{+,a}}=0$.
It is easy to see that
\begin{equation}\label{eq:ain}
  \eta_{+,a}+\omega_{+,a}=\eta^{in}_{+,a}+\omega^{in}_{+,a},
\end{equation}where $(\eta^{in}_{+,a},\omega^{in}_{+,a})$ represents the initial data. In the rest of the paper, we still use similar notation to represent the initial data.

According to (\ref{fl:aa1}) and (\ref{fl:aa2}), we get
\begin{equation}\label{eq:ptt}
  \p_{tt} \eta_{+,a}-\frac{T_+}{m_+} \p_{yy} \eta_{+,a}-\frac{4\pi e^2}{m_+}\la\xkh{-\la+\frac{4\pi e^2}{T_-}}^{-1}\eta_{+,a}=0~in~\bb{R}.
\end{equation}Next we supply system (\ref{fl:aa1})-(\ref{fl:aa3}) with the initial condition
\begin{equation}\label{eq:in0}
  \xkh{\eta_{+,a},\psi_{+,a},\omega_{+,a}}|_{t=0}=\xkh{\eta^{in}_{+,a},\psi^{in}_{+,a},\omega^{in}_{+,a}}=(0,0,0).
\end{equation}Then $\xkh{\eta_{+,a},\psi_{+,a},\omega_{+,a}}=(0,0,0)$ is the unique solution to system (\ref{fl:aa1})-(\ref{fl:aa3}) with initial condition (\ref{eq:in0}). In fact, taking the $L^2(\bb{R})$ inner product of (\ref{eq:ptt}) with $\p_t \eta_{+,a}$, then it follows from integration by parts that
\begin{equation*}
  \frac{1}{2}\frac{d}{dt}\xkh{\norm{\p_t \eta_{+,a}}^2_{L^2}+\frac{T_+}{m_+}\norm{\p_y \eta_{+,a}}^2_{L^2}
  +\frac{4\pi e^2}{m_+}\norm{\n\xkh{-\la+\frac{4\pi e^2}{T_-}}^{-\frac{1}{2}}\eta_{+,a}}^2_{L^2}}=0,
\end{equation*}which implies that $\eta_{+,a}=0$ is the unique solution to the cauchy problem, equation (\ref{eq:ptt}) with $\eta^{in}_{+,a}=0$. By virtue of (\ref{eq:ain}) and (\ref{fl:aa1}), we deduce that $\xkh{\psi_{+,a},\omega_{+,a}}=(0,0)$. Based on this fact, for simplicity of notation, we study the dynamics of $(\eta_{+},\psi_{+},\omega_{+})$ with $\xkh{\eta^{in}_{+,a},\psi^{in}_{+,a},\omega^{in}_{+,a}}=(0,0,0)$, instead of studying the dynamics for $(\eta_{+}-\eta_{+,a},\psi_{+}-\psi_{+,a},\omega_{+}-\omega_{+,a})$.

The main results for the ion dynamics are as follows.
\begin{theorem}\label{th:uqp}
Suppose that $(\eta^{in}_{+},\omega^{in}_{+}) \in H^1_x H^2_y$ and that $\psi^{in}_{+} \in H^{-\frac{1}{2}}_x L^2_y$ with $\xkh{\eta^{in}_{+,a},\psi^{in}_{+,a},\omega^{in}_{+,a}}=(0,0,0)$. Let $(\eta_{+},u_{+},\phi)$ be a smooth solution for the system \eqref{fl:ndu}-\eqref{fl:phi}. Then the following estimates hold
\begin{flalign}
  \norm{\bb{P}[u_{+}]^x(t)}_{L^2}+\norm{\phi(t)}_{L^2} &\lesssim \sqrt{\frac{m_+}{T_+}}\frac{1}{\lan{t}^\frac{1}{2}}\xkh{\norm{\sqrt{\frac{T_+}{m_+}}
  \eta^{in}_{+}}_{H^{-\frac{1}{2}}_x L^2_y}+\norm{\eta^{in}_{+}+\omega^{in}_{+}}_{H^{-\frac{1}{2}}_x H^{\frac{1}{2}}_y}}\nonumber\\
  &\quad+\frac{1}{\lan{t}}\norm{\eta^{in}_{+}+\omega^{in}_{+}}_{H^{-1}_x H^{1}_y}
  +\sqrt{\frac{m_+}{T_+}}\frac{1}{\lan{t}^\frac{1}{2}}\norm{\psi^{in}_{+}}_{H^{-\frac{1}{2}}_x H^{-1}_y},\label{fl:pux}
\end{flalign}
\begin{flalign}
  \norm{\bb{P}[u_{+}]^y(t)}_{L^2} &\lesssim \sqrt{\frac{m_+}{T_+}}\frac{1}{\lan{t}^\frac{3}{2}}\xkh{\norm{\sqrt{\frac{T_+}{m_+}}
  \eta^{in}_{+}}_{H^{-\frac{1}{2}}_x H^1_y}+\norm{\eta^{in}_{+}+\omega^{in}_{+}}_{H^{-\frac{1}{2}}_x H^{\frac{3}{2}}_y}}\nonumber\\
  &\quad+\frac{1}{\lan{t}^2}\norm{\eta^{in}_{+}+\omega^{in}_{+}}_{H^{-1}_x H^{2}_y}
  +\sqrt{\frac{m_+}{T_+}}\frac{1}{\lan{t}^\frac{3}{2}}\norm{\psi^{in}_{+}}_{H^{-\frac{1}{2}}_x L^2_y},\label{fl:puy}
\end{flalign}and
\begin{flalign}
  &\norm{\bb{Q}[u_{+}](t)}_{L^2}+\sqrt{\frac{T_+}{m_+}}\norm{\eta_{+}(t)}_{L^2}\nonumber\\
  &\qquad\lesssim \lan{t}^\frac{1}{2}\xkh{\norm{\sqrt{\frac{T_+}{m_+}}\eta^{in}_{+}}_{L^2}
  +\norm{\psi^{in}_{+}}_{H^{-1}}+\norm{\eta^{in}_{+}+\omega^{in}_{+}}_{H^1}}.\label{fl:qut}
\end{flalign}
\end{theorem}

\begin{theorem}\label{th:que}
Assume that $\xkh{\eta^{in}_{+},\omega^{in}_{+}} \in L^2_xH^{-\frac{1}{2}}_y$ and that $\psi^{in}_{+} \in H^{-\frac{3}{2}}_xH^{-2}_y$ with $\xkh{\eta^{in}_{+,a},\psi^{in}_{+,a},\omega^{in}_{+,a}}=(0,0,0)$. Let $(\eta_{+},u_{+},\phi)$ be a smooth solution for the system \eqref{fl:ndu}-\eqref{fl:phi}. Then it holds that
\begin{equation*}
  \norm{\bb{Q}[u_{+}](t)}_{L^2}+\sqrt{\frac{T_+}{m_+}}\norm{\eta_{+}(t)}_{L^2}
  \gtrsim \ca{C}^{in}_{\delta}\xkh{\eta^{in}_{+},\psi^{in}_{+},\omega^{in}_{+}}\lan{t}^{\frac{1}{2}},
\end{equation*}where the constant $\ca{C}^{in}_{\delta}\xkh{\eta^{in}_{+},\psi^{in}_{+},\omega^{in}_{+}}$ is a suitable combination of $\norm{\eta^{in}_{+}}_{L^2_xH^{-\frac{1}{2}}_y}$, $\norm{\psi^{in}_{+}}_{H^{-\frac{3}{2}}_xH^{-2}_y}$ and $\norm{\omega^{in}_{+}}_{L^2_xH^{-\frac{1}{2}}_y}$.
\end{theorem}

The above two theorems give upper and lower bounds of the linearized solutions of the Euler-Poisson system near Couette flow. In particular, the inviscid damping for the solenoidal component of the velocity is established, while the $L^2$-norm of the velocity grows as $t^{1/2}$. We also remark that it may happen that the constant $\ca{C}^{in}_{\delta}$ vanishes at some positive time $t>0$, however we can prove that the set of initial data for which the RHS of the lower bound in Theorem \ref{th:que} vanishes at some time has empty interior in any Sobolev space in which the  initial data are taken. The details are similar to that for the Euler system in \cite{pm2021} and are omitted here.

\subsection{Electron dynamics case}
To study the linear stability for the Couette flow in system (\ref{eepfl:wie})-(\ref{eepfl:e0w}), we also carry out a perturbation around the above stationary solution, \ie,
\begin{equation*}
  \wi{\eta}_{-}=\eta_{-}+\wi{\eta}_s, \quad \wi{v}_{-}=u_{-}+\wi{v}_s, \quad \wi{\phi}=\phi+\wi{\phi}_s.
\end{equation*}
Then the linearized system around the stationary solution $(\wi{\eta}_s, \wi{v}_s, \wi{\phi}_s)$ for system (\ref{eepfl:wie})-(\ref{eepfl:e0w}) can be rewritten as
\begin{flalign}
  &\p_t \eta_{-}+y\p_x \eta_{-}+(\n \d u_{-})=0,\label{eepfl:ndu}\\
  &\p_t u_{-}+y\p_x u_{-}+\begin{pmatrix}u^y_{-}\\0\end{pmatrix}+\frac{1}{m_-} \n \eta_{-}=\frac{e}{m_-}\n\phi,\label{eepfl:ptu}\\
  &\la\phi=4 \pi e \eta_{-}.\label{eepfl:phi}
\end{flalign}
Let $\psi_{-}=\n \d u_{-}$ and $\omega_{-}=\n^\perp \d u_{-}$ with $\n^\perp=(-\p_y,\p_x)^{\top}$. Then the Helmholtz projection operators are defined as
\begin{equation}
  u_{-}=(u^x_{-},u^y_{-})^{\top}=\n\la^{-1}\psi_{-}+\n^\perp\la^{-1}\omega_{-}:=\bb{Q}[u_{-}]+\bb{P}[u_{-}].\label{eepeq:pqu}
\end{equation}
The linearized system (\ref{eepfl:ndu})-(\ref{eepfl:phi}) is rewritten as (in the variables $(\eta_{-},\psi_{-},\omega_{-})$)
\begin{flalign}
  &\p_t \eta_{-}+y\p_x \eta_{-}+\psi_{-}=0,\label{eepfl:eta}\\
  &\p_t \psi_{-}+y\p_x \psi_{-}+2\p_x u^y_{-}+\frac{1}{m_-} \la \eta_{-}=\frac{4\pi e^2}{m_-}\eta_{-},\label{eepfl:psi}\\
  &\p_t \omega_{-}+y\p_x\omega_{-}-\psi_{-}=0.\label{eepfl:ome}
\end{flalign}

Just as in the ion dynamics case, we take the $x$-average for equations (\ref{eepfl:eta}), (\ref{eepfl:psi}) and (\ref{eepfl:ome}) to obtain the following system
\begin{flalign}
  &\p_t \eta_{-,a}+\psi_{-,a}=0,\label{eepfl:aa1}\\
  &\p_t \psi_{-,a}+\frac{1}{m_-} \p_{yy} \eta_{-,a}=\frac{4\pi e^2}{m_-}\eta_{-,a},\label{eepfl:aa2}\\
  &\p_t \omega_{-,a}-\psi_{-,a}=0.\label{eepfl:aa3}
\end{flalign}
Again, combining (\ref{eepfl:aa1}) with (\ref{eepfl:aa3}) yields $\p_t\xkh{\eta_{-,a}+\omega_{-,a}}=0$ and hence
\begin{equation}\label{eepeq:ain}
  \eta_{-,a}+\omega_{-,a}=\eta^{in}_{-,a}+\omega^{in}_{-,a},
\end{equation}
where $(\eta^{in}_{-,a},\omega^{in}_{-,a})$ represents the initial data. According to (\ref{eepfl:aa1}) and (\ref{eepfl:aa2}), we get
\begin{equation}\label{eepeq:ptt}
  \p_{tt} \eta_{-,a}-\frac{1}{m_-} \p_{yy} \eta_{-,a}+\frac{4\pi e^2}{m_-}\eta_{-,a}=0~in~\bb{R},
\end{equation}
which is the famous Klein-Gordon equation. The fact that at the linearized level (near the constant solution), the electron Euler-Poisson system satisfies the Klein-Gordon equation was found by Guo \cite{Guo1998} and the global irrotational flows in the large to the electron Euler-Poisson system was hence obtained in $\bb{R}^{3+1}$. For global well-posedness of the lower dimensional Euler-Poisson system, one may refer to \cite{JLZ2014,Jang2012,IP2013,LW2014}.

Similar to the ion dynamics case,  we supply system (\ref{eepfl:aa1})-(\ref{eepfl:aa3}) with the initial condition
\begin{equation}\label{eepeq:in0}
  \xkh{\eta_{-,a},\psi_{-,a},\omega_{-,a}}|_{t=0}=\xkh{\eta^{in}_{-,a},\psi^{in}_{-,a},\omega^{in}_{-,a}}=(0,0,0).
\end{equation}
From (\ref{eepeq:ain}) and the explicit representation formula for the Klein-Gordon equation, we deduce that $\xkh{\eta_{-,a},\psi_{-,a},\omega_{-,a}}=(0,0,0)$ is the unique solution to system (\ref{eepfl:aa1})-(\ref{eepfl:aa3}) with initial condition (\ref{eepeq:in0}). Based on this fact, for simplicity of notation, we study the dynamics of $(\eta_{-},\psi_{-},\omega_{-})$ with $\xkh{\eta^{in}_{-,a},\psi^{in}_{-,a},\omega^{in}_{-,a}}=(0,0,0)$, instead of studying the dynamics for $(\eta_{-}-\eta_{-,a},\psi_{-}-\psi_{-,a},\omega_{-}-\omega_{-,a})$.

The main results for the linear stability of the electron Euler-Poisson system are stated as follows.
\begin{theorem}\label{eepth:uqp}
Suppose that $(\eta^{in}_{-},\omega^{in}_{-}) \in H^1_x H^2_y$ and that $\psi^{in}_{-} \in H^{-\frac{1}{2}}_x L^2_y$ with $\xkh{\eta^{in}_{-,a},\psi^{in}_{-,a},\omega^{in}_{-,a}}=(0,0,0)$. Let $(\eta_{-},u_{-},\phi)$ be a smooth solution for the system \eqref{eepfl:ndu}-\eqref{eepfl:phi}. Then the following estimates hold
\begin{flalign}
  \norm{\bb{P}[u_{-}]^x(t)}_{L^2}+\norm{\phi(t)}_{L^2} &\lesssim \frac{\sqrt{m_-}}{\lan{t}^\frac{1}{2}}\xkh{\norm{\sqrt{\frac{1}{m_-}}
  \eta^{in}_{-}}_{H^{-\frac{1}{2}}_x L^2_y}+\norm{\eta^{in}_{-}+\omega^{in}_{-}}_{H^{-\frac{1}{2}}_x H^{\frac{1}{2}}_y}}\nonumber\\
  &\quad+\frac{1}{\lan{t}}\norm{\eta^{in}_{-}+\omega^{in}_{-}}_{H^{-1}_x H^{1}_y}
  +\frac{\sqrt{m_-}}{\lan{t}^\frac{1}{2}}\norm{\psi^{in}_{-}}_{H^{-\frac{1}{2}}_x H^{-1}_y},\label{eepfl:pux}
\end{flalign}
\begin{flalign}
  \norm{\bb{P}[u_{-}]^y(t)}_{L^2} &\lesssim \frac{\sqrt{m_-}}{\lan{t}^\frac{3}{2}}\xkh{\norm{\sqrt{\frac{1}{m_-}}
  \eta^{in}_{-}}_{H^{-\frac{1}{2}}_x H^1_y}+\norm{\eta^{in}_{-}+\omega^{in}_{-}}_{H^{-\frac{1}{2}}_x H^{\frac{3}{2}}_y}}\nonumber\\
  &\quad+\frac{1}{\lan{t}^2}\norm{\eta^{in}_{-}+\omega^{in}_{-}}_{H^{-1}_x H^{2}_y}
  +\frac{\sqrt{m_-}}{\lan{t}^\frac{3}{2}}\norm{\psi^{in}_{-}}_{H^{-\frac{1}{2}}_x L^2_y},\label{eepfl:puy}
\end{flalign}and
\begin{flalign}
  &\norm{\bb{Q}[u_{-}](t)}_{L^2}+\sqrt{\frac{1}{m_-}}\norm{\eta_{-}(t)}_{L^2}\nonumber\\
  &\qquad\lesssim \lan{t}^\frac{1}{2}\xkh{\norm{\sqrt{\frac{1}{m_-}}\eta^{in}_{-}}_{L^2}
  +\norm{\psi^{in}_{-}}_{H^{-1}}+\norm{\eta^{in}_{-}+\omega^{in}_{-}}_{H^1}}.\label{eepfl:qut}
\end{flalign}
\end{theorem}

\begin{theorem}\label{eeth:que}
Assume that $\xkh{\eta^{in}_{-},\omega^{in}_{-}} \in L^2_xH^{-\frac{1}{2}}_y$ and that $\psi^{in}_{-} \in H^{-\frac{3}{2}}_xH^{-2}_y$ with $\xkh{\eta^{in}_{-,a},\psi^{in}_{-,a},\omega^{in}_{-,a}}=(0,0,0)$. Let $(\eta_{-},u_{-},\phi)$ be a smooth solution for the system \eqref{eepfl:ndu}-\eqref{eepfl:phi}. Then we have
\begin{equation*}
  \norm{\bb{Q}[u_{-}](t)}_{L^2}+\sqrt{\frac{1}{m_-}}\norm{\eta_{-}(t)}_{L^2}
  \gtrsim \ca{C}^{in}_{\sigma}\xkh{\eta^{in}_{-},\psi^{in}_{-},\omega^{in}_{-}}\lan{t}^{\frac{1}{2}},
\end{equation*}where the constant $\ca{C}^{in}_{\sigma}\xkh{\eta^{in}_{-},\psi^{in}_{-},\omega^{in}_{-}}$ depends on Sobolev norms of the initial data.
\end{theorem}

\begin{remark}
It should be noted that at the linearized level, the dynamics of the ion Euler-Poisson system and the electron Euler-Poisson system are depicted by different linear PDE systems, \eqref{eq:ptt} for the ion dynamics and the \eqref{eepeq:ptt} for the electron dynamics. The linear dispersion relation behaves much closer to the wave dispersion for the ions and to the Klein-Gordon dispersion for the electrons. We investigate the linear stability the Couette flow for ion and electron dynamics separately below.
\end{remark}

\begin{remark}
Since the analysis in frequency space can be given by Lemma \ref{le:aet} or \ref{eeple:aet}, estimates \eqref{fl:qut} and \eqref{eepfl:qut} are also valid in the general Sobolev space $H^{-s}(s>\frac{1}{2})$.
\end{remark}

This paper is organized as follows. In the next section, we will give some notations and the definitions of Sobolev spaces. We will also write out the coordinate transform explicitly. The linear stability of the ion dynamics near the Couette flow will be analyzed in Section \ref{se:3} and the linear stability of the electron dynamics near the Couette flow will be analyzed in Section \ref{se:4}.

\section{Notations}\label{se:2}
The Fourier transform for a function $f(x,y)$ is denoted by
\begin{equation*}
  \ha{f}(k,\ze)=\frac{1}{2\pi}\iint_{\bb{T}\times\bb{R}}e^{-i\xkh{kx+\ze y}}f(x,y)dxdy.
\end{equation*}Then
\begin{equation*}
  f(x,y)=\frac{1}{2\pi}\sum_{k\in\bb{Z}}\int_{\bb{R}}e^{i\xkh{kx+\ze y}}\ha{f}(k,\ze)d\ze.
\end{equation*}The anisotropic Sobolev space is defined by
\begin{equation*}
  H^r_xH^s_y({\bb{T}\times\bb{R}})=\dkh{f:\norm{f}^2_{H^r_xH^s_y({\bb{T}\times\bb{R}})}
  =\sum_k\int\lan{k}^{2r}\lan{\ze}^{2s}\abs{\ha{f}}^2(k,\ze)d\ze<+\infty},
\end{equation*}where $\lan{a}=(1+a^2)^{1/2}$ for $a\in\bb{R}$. In the same way, the usual Sobolev space is defined as
\begin{equation*}
  H^s({\bb{T}\times\bb{R}})=\dkh{f:\norm{f}^2_{H^s({\bb{T}\times\bb{R}})}
  =\sum_k\int\lan{k,\ze}^{2s}\abs{\ha{f}}^2(k,\ze)d\ze<+\infty}.
\end{equation*}Here $\lan{a,b}=(1+a^2+b^2)^{1/2}$ for $a,b\in\bb{R}$.

In order to deal with the transport term in system (\ref{fl:eta})-(\ref{fl:ome}) or (\ref{eepfl:eta})-(\ref{eepfl:ome}), we introduce a set of coordinate transformations
\begin{equation}\label{eq:xyt}
  X=x-yt,~Y=y.
\end{equation}So the differential operators can be represented as
\begin{equation}\label{eq:ngn}
  \p_x=\p_X,~\p_y=\p_Y-t\p_X,~\la=:\la_L=\p_{XX}+\xkh{\p_Y-t\p_X}^2.
\end{equation}
We denote the symbol for $-\la_L$ by
\begin{equation}\label{eq:kt2}
  \alpha(t,k,\ze)=k^2+(\ze-kt)^2.
\end{equation}Then the symbol of the operator $2\p_X(\p_Y-t\p_X)$ is
\begin{equation}\label{eq:pta}
  \p_t\alpha(t,k,\ze)=-2k(\ze-kt).
\end{equation}

\section{Analysis for the ion dynamics}\label{se:3}
In this section, we consider the dynamics of the system \eqref{fl:eta}-\eqref{fl:ome}. Using the coordinate transformations \eqref{eq:xyt}, we define the functions
\begin{flalign}
  &\Pi_{+}(t,X,Y)=\eta_{+}(t,X+tY,Y),\label{fl:pit}\\
  &\Psi_{+}(t,X,Y)=\psi_{+}(t,X+tY,Y),\label{fl:sit}\\
  &\Gamma_{+}(t,X,Y)=\omega_{+}(t,X+tY,Y).\label{fl:gam}
\end{flalign}
Set $F_{+}(t,X,Y):=\Pi_{+}(t,X,Y)+\Gamma_{+}(t,X,Y)$. Combining (\ref{fl:eta}) with (\ref{fl:ome}), we have $\p_t F_{+}=0$, which yields
\begin{equation*}
  \Gamma_{+}(t,X,Y)=F^{in}_{+}(X,Y)-\Pi_{+}(t,X,Y).
\end{equation*}Here $F^{in}_{+}=\eta^{in}_{+}+\omega^{in}_{+}$. According to (\ref{eq:pqu}), we obtain
\begin{flalign}
  U^y_{+}&=\xkh{\p_Y-t\p_X}\la^{-1}_L\Psi_{+}+\p_X\la^{-1}_L\Gamma_{+}\nonumber\\
  &=\xkh{\p_Y-t\p_X}\la^{-1}_L\Psi_{+}+\p_X\la^{-1}_L F^{in}_{+}-\p_X\la^{-1}_L \Pi_{+} \nonumber.
\end{flalign}In consequence, we can rewrite system (\ref{fl:eta})-(\ref{fl:ome}) as
\begin{flalign}
  &\p_t \Pi_{+}+\Psi_{+}=0,\label{fl:ptp}\\
  &\p_t \Psi_{+}+2\p_{XX}\la^{-1}_L F^{in}_{+}+2\p_X\xkh{\p_Y-t\p_X}\la^{-1}_L\Psi_{+}\nonumber\\
  &-\xkh{2\p_{XX}\la^{-1}_L-\frac{T_+}{m_+}\la_L}\Pi_{+}
  +\frac{4\pi e^2}{m_+}\la_L\xkh{-\la_L+\frac{4\pi e^2}{T_-}}^{-1}\Pi_{+}=0.\label{fl:2me}
\end{flalign}

Taking the Fourier transform on system (\ref{fl:ptp})-(\ref{fl:2me}), we get the following system
\begin{flalign}
  &\p_t \ha{\Pi_{+}}+\ha{\Psi_{+}}=0,\label{fl:aps}\\
  &\p_t \ha{\Psi_{+}}+\frac{2k^2}{\alpha}\ha{F^{in}_{+}}-\frac{\p_t \alpha}{\alpha}\ha{\Psi_{+}}
  -\xkh{\frac{2k^2}{\alpha}+\frac{T_+\alpha}{m_+}
  +\frac{4 \pi e^2}{m_+}\frac{\alpha}{\alpha+\frac{4\pi e^2}{T_-}}}\ha{\Pi_{+}}=0.\label{fl:haf}
\end{flalign}In order to study system (\ref{fl:aps})-(\ref{fl:haf}), we need to find a suitable symmetrization for this system. Define
\begin{equation}\label{eq:atx}
  A(t)=\xkh{A_1(t),A_2(t)}^{\top}
  =\xkh{\sqrt{\frac{T_+}{m_+}}\frac{\ha{\Pi_{+}}(t)}{\alpha^{1/4}},\frac{\ha{\Psi_{+}}(t)}{\alpha^{3/4}}}^{\top}.
\end{equation}By a delicate calculation, $A(t)$ satisfy a non-autonomous 2D dynamical system
\begin{equation}\label{eq:ddt}
\begin{cases}
  \frac{d}{dt}A(t)=L_{+}(t)A(t)+M_{+}(t)\ha{F^{in}_{+}},\\
  A(0)=A^{in},
\end{cases}
\end{equation}where
\begin{flalign}
  &L_{+}(t)=\begin{bmatrix}-\frac{1}{4}\alpha^{-1}\p_t\alpha & -\sqrt{\frac{T_+}{m_+}}\alpha^{1/2}\\
  \frac{4\pi e^2}{\sqrt{m_+T_+}}\frac{\alpha^{1/2}}{\alpha+\frac{4 \pi e^2}{T_-}}+\sqrt{\frac{m_+}{T_+}}\frac{2k^2}{\alpha^{3/2}}
  +\sqrt{\frac{T_+}{m_+}}\alpha^{1/2} & \frac{1}{4}\alpha^{-1}\p_t\alpha\end{bmatrix},\\
  &M_{+}(t)=\xkh{0,-\frac{2k^2}{\alpha^{7/4}}}^{\top},~A^{in}=\xkh{\sqrt{\frac{T_+}{m_+}}\frac{\ha{\Pi^{in}_{+}}}{\xkh{k^2+\ze^2}^{1/4}},
  \frac{\ha{\Psi^{in}_{+}}}{\xkh{k^2+\ze^2}^{3/4}}}^{\top}.
\end{flalign}

By virtue of Duhamel's formula, the solution $A(t)$ to system (\ref{eq:ddt}) is given by
\begin{equation}\label{eq:srt}
  A(t)=S_L(t,0)A^{in}+\int^t_0 S_L(t,s)M_{+}(s)\ha{F^{in}_{+}}ds.
\end{equation}
Here, $S_L$ denotes the solution operator, satisfying the group property
\begin{equation*}
  S_L(t,0)S_L(0,s)=S_L(t,s)
\end{equation*}for any $t,s>0$. In consequence, it suffices to study the homogeneous problem of system (\ref{eq:ddt}). For this purpose, we denote by $A(t)$ a solution to system (\ref{eq:ddt}) with $\ha{F^{in}_{+}}=0$. Let
\begin{equation}
  \ca{E}_{+}(t)=\xkh{\sqrt{\frac{p_1}{m_1}}|A_1|^2}(t)
  +2\xkh{\frac{h_1}{\sqrt{m_1p_1}}\Re\xkh{A_1 \bar{A_2}}}(t)+\xkh{\sqrt{\frac{m_1}{p_1}}|A_2|^2}(t),
\end{equation}where
\begin{flalign}
  &h_1(t)=\frac{1}{4}\alpha^{-1}\p_t\alpha,~m_1(t)=\sqrt{\frac{T_+}{m_+}}\alpha^{1/2},\label{fl:hmt}\\
  &p_1(t)=\frac{4\pi e^2}{\sqrt{m_+T_+}}\frac{\alpha^{1/2}}{\alpha+\frac{4 \pi e^2}{T_-}}
  +\sqrt{\frac{m_+}{T_+}}\frac{2k^2}{\alpha^{3/2}}+\sqrt{\frac{T_+}{m_+}}\alpha^{1/2}.\label{fl:pmt}
\end{flalign}
Then we can get the upper and lower bounds for $\ca{E}_+(t)$ and the solution $A(t)$.
\begin{lemma}\label{le:aet}
There exists positive constants $C_1$, $C^\delta_1$, $C_2$, $C^\delta_2$  that do not depend on $k$ and $\ze$ such that
\begin{equation}\label{eq:c1e}
  C_1\ca{E}_{+}(0) \leq \ca{E}_{+}(t) \leq C^\delta_1\ca{E}_{+}(0),
\end{equation}and
\begin{equation}\label{eq:c2a}
  C_2\abs{A^{in}} \leq \abs{A(t)} \leq C^\delta_2\abs{A^{in}}.
\end{equation}
\end{lemma}
\begin{proof}
Define
\begin{gather}
  \l_{+}=\sqrt{\frac{p_1}{m_1}}=\xkh{1+\frac{2m_+k^2}{T_+\alpha^2}+\frac{4\pi e^2}{T_+}
  \frac{1}{\alpha+\frac{4\pi e^2}{T_-}}}^{1/2},\label{ga:lpm}\\
  \g_{+}=\sqrt{m_1p_1}=\xkh{\frac{4\pi e^2}{m_+}\frac{\alpha}{\alpha+\frac{4\pi e^2}{T_-}}
  +\frac{2k^2}{\alpha}+\frac{T_+\alpha}{m_+}}^{1/2}.\label{ga:gmp}
\end{gather}From the definition of $\l_{+}$ we get that
\begin{equation}\label{eq:lme}
  1 \leq (\l_{+})^2 \leq 1+\frac{4\pi e^2}{T_+}+\frac{2m_+}{T_+}.
\end{equation}Moreover, we have
\begin{flalign}
  \frac{|h_{1}|}{\g_{+}}&=\frac{|\p_t \alpha|}{4\alpha}\xkh{\frac{4\pi e^2}{m_+}\frac{\alpha}{\alpha
  +\frac{4\pi e^2}{T_-}}+\frac{2k^2}{\alpha}+\frac{T_+\alpha}{m_+}}^{-1/2}\nonumber\\
  &=\frac{|\p_t \alpha|}{4\alpha}\xkh{\frac{m_+}{T_+\alpha}}^{1/2}\xkh{1+\frac{2m_+k^2}{T_+\alpha^2}
  +\frac{4\pi e^2}{T_+}\frac{1}{\alpha+\frac{4\pi e^2}{T_-}}}^{-1/2}\nonumber\\
  &\leq\frac{|k|}{2\sqrt{\alpha}}\xkh{\frac{m_+}{T_+\alpha}}^{1/2}
  \xkh{\frac{T_+\alpha^2}{2m_+k^2}}^{1/2}\leq\frac{\sqrt{2}}{2},\label{fl:hg2}
\end{flalign}
where we have used the fact that $|\p_t \alpha|\leq2|k|\sqrt{\alpha}$.

Setting
\begin{equation*}
  \wi{\ca{E}}_{+}(t)=\xkh{\l_{+}|A_1|^2}(t)+\xkh{\frac{1}{\l_{+}}|A_2|^2}(t).
\end{equation*}Then from (\ref{fl:hg2}) it can be deduced that
\begin{equation}\label{eq:E22}
\xkh{1-\frac{\sqrt{2}}{2}}\wi{\ca{E}}_{+}(t)\leq\ca{E}_{+}(t)\leq \xkh{1+\frac{\sqrt{2}}{2}}\wi{\ca{E}}_{+}(t).
\end{equation}
Obviously, the coerciveness of $\ca{E}_{+}(t)$ is ensured by the above inequality and (\ref{eq:lme}).

According to (\ref{fl:hmt}), (\ref{fl:pmt}), (\ref{ga:lpm}) and (\ref{ga:gmp}), system (\ref{eq:ddt}) becomes
\begin{gather}
  \l_{+}\frac{d}{dt}A_1=-h_{1}\l_{+} A_1-\g_{+} A_2,\label{ga:la1}\\
  \frac{1}{\l_{+}}\frac{d}{dt}A_2=\g_{+} A_1+\frac{h_{1}}{\l_{+}}A_2.\label{ga:la2}
\end{gather}
Multiplying equations (\ref{ga:la1}) and (\ref{ga:la2}) by $\bar{A}_1$ and $\bar{A}_2$, respectively, and similarly for their conjugate, we get
\begin{equation}\label{eq:dl2}
  \frac{\l_{+}}{2}\frac{d}{dt}|A_1|^2+\frac{1}{2\l_{+}}\frac{d}{dt}|A_2|^2=-h_{1}\l_{+}|A_1|^2+\frac{h_{1}}{\l_{+}}|A_2|^2.
\end{equation}
Note the fact that
\begin{equation}\label{eq:hgd}
  \frac{h_{1}}{\g_{+}}\frac{d}{dt}\Re(A_1\bar{A}_2)=h_{1}\l_{+}|A_1|^2-\frac{h_{1}}{\l_{+}}|A_2|^2.
\end{equation}
Adding (\ref{eq:dl2}) to (\ref{eq:hgd}), we obtain
\begin{equation*}
  \l_{+}\frac{d}{dt}|A_1|^2+\frac{2h_{1}}{\g_{+}}\frac{d}{dt}\Re(A_1\bar A_2)+\frac{1}{\l_{+}}\frac{d}{dt}|A_2|^2=0,
\end{equation*}which yields
\begin{equation*}
  \frac{d\ca{E}_{+}}{dt}=\l_{+}|A_1|^2\frac{d}{dt}\xkh{\log\l_{+}}+2\Re(A_1\bar{A_2})\frac{d}{dt}\xkh{\frac{h_{1}}{\g_{+}}}
  -\frac{1}{\l_{+}}|A_2|^2\frac{d}{dt}\xkh{\log\l_{+}}.
\end{equation*}Using Young's inequality, from (\ref{eq:E22}) we conclude that
\begin{flalign}
  \frac{d\ca{E}_{+}}{dt}&\leq\xkh{\abs{\frac{d}{dt}\xkh{\frac{h_{1}}{\g_{+}}}}+\abs{\frac{d}{dt}\xkh{\log\l_{+}}}}\wi{\ca{E}}_{+}\nonumber\\
  &\leq \xkh{2+\sqrt{2}}\xkh{\abs{\frac{d}{dt}\xkh{\frac{h_{1}}{\g_{+}}}}+\abs{\frac{d}{dt}\xkh{\log\l_{+}}}}\ca{E}_{+}.\label{fl:dca}
\end{flalign}An argument similar to that of (\ref{fl:dca}) gives
\begin{equation}\label{fl:geq}
  \frac{d\ca{E}_{+}}{dt}\geq-\xkh{2+\sqrt{2}}\xkh{\abs{\frac{d}{dt}\xkh{\frac{h_{1}}{\g_{+}}}}+\abs{\frac{d}{dt}\xkh{\log\l_{+}}}}\ca{E}_{+}.
\end{equation}

For the purpose of applying Gr\"{o}nwall's inequality, we need to estimate the integrals
 $$\int^{+\infty}_0\abs{\frac{d}{dt}\xkh{\frac{h_{1}}{\g_{+}}}}d\vt \ \ \ and \ \ \ \int^{+\infty}_0\abs{\frac{d}{dt}\xkh{\log\l_{+}}}d\vt.$$
First of all, we deal with the first one. By a direct calculation, we have
\begin{equation*}
  \frac{d}{dt}\xkh{\frac{h_{1}}{\g_{+}}}=\frac{\ca G}{8m_+(T_-)^2(\g_{+})^3\alpha^3\xkh{\alpha+\frac{4 \pi e^2}{T_-}}^3},
\end{equation*}where
\begin{flalign*}
  \ca G=&16\pi e^2(T_-)^2k^2\alpha^5+\xkh{8m_+T_-k^4\alpha^3+4T_+T_-k^2\alpha^5}\xkh{T_-\alpha+4\pi e^2}\\
  &+128\pi^2e^4T_-k^2\alpha^4+\xkh{64\pi e^2m_+k^4\alpha^2+32\pi e^2T_+k^2\alpha^4}\xkh{T_-\alpha+4\pi e^2}\\
  &+256\pi^3e^6k^2\alpha^3+\xkh{128\pi^2e^4m_+k^4\alpha+64\pi^2e^4T_+k^2\alpha^3}\xkh{\alpha+\frac{4 \pi e^2}{T_-}}\\
  &-8\pi e^2(T_-)^2\alpha^4(\p_t\alpha)^2-\zkh{4m_+T_-k^2\alpha^2(\p_t\alpha)^2+2T_+T_-\alpha^4(\p_t\alpha)^2}\xkh{T_-\alpha+4\pi e^2}\\
  &-64\pi^2e^4T_-\alpha^3(\p_t\alpha)^2-\zkh{32\pi e^2m_+k^2\alpha(\p_t\alpha)^2+16\pi e^2T_+\alpha^3(\p_t\alpha)^2}\xkh{T_-\alpha+4\pi e^2}\\
  &-128\pi^3e^6\alpha^2(\p_t\alpha)^2-\zkh{64\pi^2e^4m_+k^2(\p_t\alpha)^2+32\pi^2e^4T_+\alpha^2(\p_t\alpha)^2}\xkh{\alpha+\frac{4 \pi e^2}{T_-}}\\
  &+\zkh{16 \pi e^2m_+k^2\alpha(\p_t\alpha)^2-8 \pi e^2T_+\alpha^3(\p_t\alpha)^2+\frac{32\pi^2e^4m_+k^2}{T_-}(\p_t\alpha)^2}\xkh{T_-\alpha+4\pi e^2}\\
  &+\zkh{\xkh{2m_+T_-k^2-16\pi^2e^4-\frac{16\pi^2e^4T_+}{T_-}}\alpha^2(\p_t\alpha)^2-T_+T_-\alpha^4(\p_t\alpha)^2}\xkh{T_-\alpha+4\pi e^2}
\end{flalign*}
is a polynomial in time of order 12. Denote by $t_i(i=1,2,...,12)$ the possible positive roots for $\ca G$. We set $i_0=0$ provided that $\ca{G}(0) \leq 0$ and take $i_0=1$ for the other cases. Assume that $t_0=0$ and that $t_{13}=+\infty$. By virtue of (\ref{fl:hg2}), we infer
\begin{flalign}
  \int^{+\infty}_0\abs{\frac{d}{dt}\xkh{\frac{h_{1}}{\g_{+}}}}d\vt
  &=\sum^{13}_{i=1}(-1)^{i+i_0}\xkh{\frac{h_{1}}{\g_{+}}(t_i)-\frac{h_{1}}{\g_{+}}(t_{i-1})}\nonumber\\
  &\leq2\sum^{13}_{i=0}\frac{|h_{1}|}{\g_{+}}(t_i) \leq 14\sqrt{2}.\label{fl:b13}
\end{flalign}For the second integral term, we deduce from (\ref{eq:lme}) that
\begin{flalign}
  \int^{+\infty}_0\abs{\frac{d}{dt}\xkh{\log\l_{+}}}d\vt&=\frac{1}{2}\int^{\frac{\ze}{k}}_0\abs{\frac{d}{dt}\xkh{\log(\l_{+})^2}}d\vt
  +\frac{1}{2}\int^{+\infty}_{\frac{\ze}{k}}\abs{\frac{d}{dt}\xkh{\log(\l_{+})^2}}d\vt\nonumber\\
  &\leq\frac{1}{2}\zkh{\xkh{\log(\l_{+})^2}\xkh{\frac{\ze}{k}}-\xkh{\log(\l_{+})^2}\xkh{0}}\nonumber\\
  &\quad+\frac{1}{2}\zkh{\xkh{\log(\l_{+})^2}\xkh{\frac{\ze}{k}}-\xkh{\log(\l_{+})^2}\xkh{+\infty}}\nonumber\\
  &\leq\log\xkh{1+\frac{4\pi e^2}{T_+}+\frac{2m_+}{T_+}},\label{fl:me0}
\end{flalign}where we have used the fact that
\begin{flalign*}
  \frac{1}{(\l_{+})^2}\frac{d (\l_{+})^2}{dt}
  &=(T_+)^{-1}(\l_{+})^{-2}\alpha^{-3}\xkh{\alpha+\frac{4\pi e^2}{T_-}}^{-2}\\
  &\quad\times\zkh{-4m_+k^2\xkh{\alpha+\frac{4\pi e^2}{T_-}}^2-4\pi e^2\alpha^3}\p_t\alpha=0
\end{flalign*}
has a unique root $t=\ze/k$. Applying Gr\"{o}nwall's inequality to (\ref{fl:dca}) and (\ref{fl:geq}), respectively, from (\ref{fl:b13}) and (\ref{fl:me0}) we obtain (\ref{eq:c1e}).

Next, it remains to prove (\ref{eq:c2a}). Owing to (\ref{eq:c1e}) and (\ref{eq:E22}), it can be seen that
\begin{equation}\label{eq:cae}
  C'_1\wi{\ca{E}}_{+}(0) \leq \wi{\ca{E}}_{+}(t) \leq C^{\delta'}_1\wi{\ca{E}}_{+}(0).
\end{equation}By virtue of (\ref{eq:lme}), we get
\begin{equation}\label{eq:wic}
  \xkh{1+\frac{4\pi e^2}{T_+}+\frac{2m_+}{T_+}}^{-1}\wi{\ca{E}}_{+}(t) \leq |A(t)|^2
  \leq \xkh{1+\frac{4\pi e^2}{T_+}+\frac{2m_+}{T_+}}\wi{\ca{E}}_{+}(t).
\end{equation}Combining (\ref{eq:cae}) with (\ref{eq:wic}) yields (\ref{eq:c2a}). This completes the proof.
\end{proof}

Based on Lemma \ref{le:aet}, we give the proof of Theorem \ref{th:uqp}.
\begin{proof}[Proof of Theorem \ref{th:uqp}]
Due to Lemma \ref{le:aet}, the solution $A(t)$ of system (\ref{eq:ddt}) can be bounded as
\begin{flalign}
  |A(t)|&=\abs{S_L(t,0)A^{in}+\int^t_0S_L(t,s)M_{+}(s)\ha{F^{in}_{+}}ds}\nonumber\\
  &\lesssim|A^{in}|+\abs{\ha{F^{in}_{+}}}\int^t_0\frac{2k^2}{\alpha^{7/4}}ds\nonumber\\
  &\lesssim|A^{in}|+\abs{\ha{F^{in}_{+}}}\int^\infty_0\frac{1}{|k|^{3/2}(1+(\ze/k-s)^2)^{7/4}}ds\nonumber\\
  &\lesssim|A^{in}|+\abs{\ha{F^{in}_{+}}}\label{fl:ata}
\end{flalign}for any $t\geq0$. By the Helmholtz decomposition and the Plancherel's theorem, we get from (\ref{fl:ata}) that
\begin{flalign}
  \norm{\bb{P}[u_{+}]^x(t)}^2_{L^2}&+\norm{\phi(t)}^2_{L^2}=\norm{\xkh{\p_y\la^{-1}\omega_{+}}(t)}^2_{L^2}
  +\norm{4\pi e\xkh{-\la+\frac{4\pi e^2}{T_-}}^{-1}\eta_{+}(t)}^2_{L^2}\nonumber\\
  &=\sum_k\int\frac{(\ze-kt)^2}{\alpha^2}\abs{\ha{\Gamma_{+}}(t)}^2d\ze
  +\sum_k\int16\pi^2e^2\xkh{\alpha+\frac{4\pi e^2}{T_-}}^{-2}\abs{\ha{\Pi_{+}}(t)}^2d\ze\nonumber\\
  &\lesssim\sum_k\int\xkh{\frac{(\ze-kt)^2}{\alpha^2}\abs{\ha{F^{in}_{+}}}^2
  +\frac{m_+(\ze-kt)^2}{T_+\alpha^{3/2}}\abs{A_1(t)}^2+\frac{m_+}{T_+\alpha^{3/2}}\abs{A_1(t)}^2}d\ze\nonumber\\
  &\lesssim\sum_k\int\xkh{\alpha^{-1}\abs{\ha{F^{in}_{+}}}^2+\frac{m_+}{T_+\alpha^{1/2}}\abs{A^{in}}^2
  +\frac{m_+}{T_+\alpha^{1/2}}\abs{\ha{F^{in}_{+}}}^2}d\ze\nonumber\\
  &\lesssim\sum_k\int\xkh{\frac{\lan{\ze}^2}{\lan{t}^2\lan{k}^2}\abs{\ha{F^{in}_{+}}}^2
  +\frac{m_+\lan{\ze}}{T_+\lan{t}\lan{k}}\abs{A^{in}}^2+\frac{m_+\lan{\ze}}{T_+\lan{t}\lan{k}}\abs{\ha{F^{in}_{+}}}^2}d\ze\nonumber\\
  &\lesssim\sum_k\int\xkh{\frac{\lan{\ze}^2}{\lan{t}^2\lan{k}^2}\abs{\ha{\eta^{in}_{+}}+\ha{\omega^{in}_{+}}}^2
  +\frac{m_+\lan{\ze}}{T_+\lan{t}\lan{k}}\abs{\ha{\eta^{in}_{+}}+\ha{\omega^{in}_{+}}}^2}d\ze\nonumber\\
  &\quad+\sum_k\int\frac{m_+\lan{\ze}}{T_+\lan{t}\lan{k}}\xkh{\frac{T_+}{m_+}\frac{\abs{\ha{\eta^{in}_{+}}}^2}{\xkh{1+\ze^2}^{1/2}}
  +\frac{\abs{\ha{\psi^{in}_{+}}}^2}{\xkh{1+\ze^2}^{3/2}}}d\ze,
\end{flalign}where we have used the fact that $\lan{t}\lesssim\lan{\ze/k-t}\lan{\ze/k}$. Thanks to the definition of the anisotropic Sobolev space, from the above inequality we deduce (\ref{fl:pux}). Moreover, a similar argument as that of (\ref{fl:pux}) yields (\ref{fl:puy}).

Using the Helmholtz decomposition and the Plancherel's theorem, we obtain
\begin{flalign}
  \norm{\bb{Q}[u_{+}](t)}^2_{L^2}+\frac{T_+}{m_+}\norm{\eta_{+}(t)}^2_{L^2}
  &=\norm{\xkh{\n\la^{-1}\psi_{+}}(t)}^2_{L^2}+\frac{T_+}{m_+}\norm{\eta_{+}(t)}^2_{L^2}\nonumber\\
  &=\norm{\xkh{\p_x\la^{-1}\psi_{+}}(t)}^2_{L^2}+\norm{\xkh{\p_y\la^{-1}\psi_{+}}(t)}^2_{L^2}+\frac{T_+}{m_+}\norm{\eta_{+}(t)}^2_{L^2}\nonumber\\
  &=\sum_k\int\xkh{\frac{1}{k^2+(\ze-kt)^2}\abs{\ha{\Psi_{+}}(t)}^2+\frac{T_+}{m_+}\abs{\ha{\Pi_{+}}(t)}^2}d\ze.\label{fl:ut2}
\end{flalign}Note that the fact that $k^2+(\ze-kt)^2\lesssim\lan{t}^2\lan{k,\ze}^2$. From (\ref{fl:ata}) and (\ref{fl:ut2}), it follows that
\begin{flalign}
  \norm{\bb{Q}[u_{+}](t)}^2_{L^2}+\frac{T_+}{m_+}\norm{\eta_{+}(t)}^2_{L^2}
  &=\sum_k\int\xkh{\frac{1}{\alpha}\abs{\ha{\Psi_{+}}(t)}^2+\frac{T_+}{m_+}\abs{\ha{\Pi_{+}}(t)}^2}d\ze\nonumber\\
  &=\sum_k\int\alpha^\frac{1}{2}\xkh{\abs{\alpha^{-\frac{3}{4}}\ha{\Psi_{+}}(t)}^2
  +\abs{\sqrt{\frac{T_+}{m_+}}\alpha^{-\frac{1}{4}}\ha{\Pi_{+}}(t)}^2}d\ze\nonumber\\
  &=\sum_k\int\alpha^\frac{1}{2}\abs{A(t)}^2d\ze\nonumber\\
  &\lesssim \lan{t}\xkh{\norm{\sqrt{\frac{T_+}{m_+}}\eta^{in}_{+}}^2_{L^2}
  +\norm{\psi^{in}_{+}}^2_{H^{-1}}+\norm{\eta^{in}_{+}+\omega^{in}_{+}}^2_{H^1}},\label{fl:bbq}
\end{flalign}which gives the estimate (\ref{fl:qut}).
\end{proof}

Now we prove Theorem \ref{th:que}.
\begin{proof}[Proof of Theorem \ref{th:que}]
Set
\begin{equation*}
\ca{R}\xkh{t,A^{in},F^{in}_{+}}=A^{in}+\int^t_0S_L(0,s)M_{+}(s)\ha{F^{in}_{+}}ds.
\end{equation*}According to (\ref{eq:srt}), we have
\begin{equation*}
  A(t)=S_L(t,0)A^{in}+\int^t_0S_L(t,s)M_{+}(s)\ha{F^{in}_{+}}ds
  =S_L(t,0)\ca{R}\xkh{t,A^{in},F^{in}_{+}}.
\end{equation*}By virtue of Lemma \ref{le:aet}, we obtain
\begin{equation*}
  |A(t)| \geq C|\ca{R}\xkh{t,A^{in},F^{in}_{+}}|.
\end{equation*}Using the identity in (\ref{fl:bbq}), we get from the above inequality that
\begin{flalign*}
  \norm{\bb{Q}[u_{+}](t)}^2_{L^2}+\frac{T_+}{m_+}\norm{\eta_{+}(t)}^2_{L^2}
  &=\sum_k\int\alpha^\frac{1}{2}\abs{A(t)}^2d\ze\\
  &\gtrsim\sum_k\int\alpha^\frac{1}{2}|\ca{R}\xkh{t,A^{in},F^{in}_{+}}|^2d\ze\\
  &\gtrsim\sum_k\int\lan{\ze-kt}|\ca{R}\xkh{t,A^{in},F^{in}_{+}}|^2d\ze\\
  &\gtrsim\sum_k\int\frac{\lan{kt}}{\lan{\ze}}|\ca{R}\xkh{t,A^{in},F^{in}_{+}}|^2d\ze,
\end{flalign*}where we have used the facts that $\alpha^\frac{1}{2}\geq\lan{\ze-kt}$ and $\lan{\ze-kt}\lan{\ze}\gtrsim\lan{kt}$. The definition of the anisotropic Sobolev space gives
\begin{equation*}
  \norm{\bb{Q}[u_{+}](t)}^2_{L^2}+\frac{T_+}{m_+}\norm{\eta_{+}(t)}^2_{L^2}
  \gtrsim\lan{t}\norm{\ca{R}\xkh{t,A^{in},F^{in}_{+}}}^2_{L^2_xH^{-\frac{1}{2}}_y}.
\end{equation*}This completes the proof.
\end{proof}

\section{Analysis for the electron dynamics}\label{se:4}
In this section, we consider the dynamics of the electron dynamics system \eqref{eepfl:eta}-\eqref{eepfl:ome}. Also using coordinate transformations \eqref{eq:xyt}, we define the functions
\begin{flalign}
  &\Pi_{-}(t,X,Y)=\eta_{-}(t,X+tY,Y),\label{eepfl:pit}\\
  &\Psi_{-}(t,X,Y)=\psi_{-}(t,X+tY,Y),\label{eepfl:sit}\\
  &\Gamma_{-}(t,X,Y)=\omega_{-}(t,X+tY,Y).\label{eepfl:gam}
\end{flalign}
Set $F_{-}(t,X,Y):=\Pi_{-}(t,X,Y)+\Gamma_{-}(t,X,Y)$. Combining (\ref{eepfl:eta}) with (\ref{eepfl:ome}), we have $\p_t F_{-}=0$, which yields
\begin{equation*}
  \Gamma_{-}(t,X,Y)=F^{in}_{-}(X,Y)-\Pi_{-}(t,X,Y),
\end{equation*}where $F^{in}_{-}=\eta^{in}_{-}+\omega^{in}_{-}$. According to (\ref{eepeq:pqu}), we obtain
\begin{flalign}
  U^y_{-}&=\xkh{\p_Y-t\p_X}\la^{-1}_L\Psi_{-}+\p_X\la^{-1}_L\Gamma_{-}\nonumber\\
  &=\xkh{\p_Y-t\p_X}\la^{-1}_L\Psi_{-}+\p_X\la^{-1}_L F^{in}_{-}-\p_X\la^{-1}_L \Pi_{-} \nonumber.
\end{flalign}In consequence, we can rewrite system (\ref{eepfl:eta})-(\ref{eepfl:ome}) as
\begin{flalign}
  &\p_t \Pi_{-}+\Psi_{-}=0,\label{eepfl:ptp}\\
  &\p_t \Psi_{-}+2\p_{XX}\la^{-1}_L F^{in}_{-}+2\p_X\xkh{\p_Y-t\p_X}\la^{-1}_L\Psi_{-}\nonumber\\
  &-\xkh{2\p_{XX}\la^{-1}_L-\frac{1}{m_{-}}\la_L+\frac{4 \pi e^2}{m_{-}}}\Pi_{-}=0.\label{eepfl:2me}
\end{flalign}

Taking the Fourier transform on system (\ref{eepfl:ptp})-(\ref{eepfl:2me}), we get the following system
\begin{flalign}
  &\p_t \ha{\Pi_{-}}+\ha{\Psi_{-}}=0,\label{eepfl:aps}\\
  &\p_t \ha{\Psi_{-}}+\frac{2k^2}{\alpha}\ha{F^{in}_{-}}-\frac{\p_t \alpha}{\alpha}\ha{\Psi_{-}}
  -\xkh{\frac{2k^2}{\alpha}+\frac{\alpha}{m_{-}}+\frac{4 \pi e^2}{m_{-}}}\ha{\Pi_{-}}=0.\label{eepfl:haf}
\end{flalign}
In order to study system (\ref{eepfl:aps})-(\ref{eepfl:haf}), we need to find a suitable symmetrization for this system. Define
\begin{equation}\label{eepeq:atx}
  B(t)=\xkh{B_1(t),B_2(t)}^{\top}=\xkh{\frac{\ha{\Pi_{-}}(t)}{\sqrt{m_{-}}\alpha^{1/4}},\frac{\ha{\Psi_{-}}(t)}{\alpha^{3/4}}}^{\top}.
\end{equation}
By a delicate calculation, $B(t)$ satisfy a non-autonomous 2D dynamical system
\begin{equation}\label{eepeq:ddt}
\begin{cases}
  \frac{d}{dt}B(t)=L_{-}(t)B(t)+M_{-}(t)\ha{F^{in}_{-}},\\
  B(0)=B^{in},
\end{cases}
\end{equation}
where
\begin{flalign}
  &L_{-}(t)=\begin{bmatrix}-\frac{1}{4}\alpha^{-1}\p_t\alpha & -\sqrt{\frac{1}{m_{-}}}\alpha^{1/2}\\
  \frac{4\pi e^2}{\sqrt{m_{-}}\alpha^{1/2}}+\frac{2k^2\sqrt{m_{-}}}{\alpha^{3/2}}
  +\sqrt{\frac{1}{m_{-}}}\alpha^{1/2} & \frac{1}{4}\alpha^{-1}\p_t\alpha\end{bmatrix},\\
  &M_{-}(t)=\xkh{0,-\frac{2k^2}{\alpha^{7/4}}}^{\top},~B^{in}=\xkh{\frac{\ha{\Pi^{in}_{-}}}{\sqrt{m_{-}}\xkh{k^2+\ze^2}^{1/4}},
  \frac{\ha{\Psi^{in}_{-}}}{\xkh{k^2+\ze^2}^{3/4}}}^{\top}.
\end{flalign}

By virtue of Duhamel's formula, the solution $B(t)$ to system (\ref{eepeq:ddt}) is given by
\begin{equation}\label{eepeq:srt}
  B(t)=S_L(t,0)B^{in}+\int^t_0S_L(t,s)M_{-}(s)\ha{F^{in}_{-}}ds.
\end{equation}Here, $S_L$ denotes the solution operator, satisfying the group property
\begin{equation*}
  S_L(t,0)S_L(0,s)=S_L(t,s)
\end{equation*}for any $t,s>0$. In consequence, it suffices to study the homogeneous problem of system (\ref{eepeq:ddt}).

\begin{lemma}\label{eeple:aet}
Denote by $B(t)$ a solution to system (\ref{eepeq:ddt}) with $\ha{F^{in}_{-}}=0$. Let
\begin{equation}
  \ca{E}_{-}(t)=\xkh{\sqrt{\frac{p_{2}}{m_{2}}}|B_1|^2}(t)
  +2\xkh{\frac{h_{2}}{\sqrt{m_{2}p_{2}}}\Re\xkh{B_1 \bar{B_2}}}(t)
  +\xkh{\sqrt{\frac{m_{2}}{p_{2}}}|B_2|^2}(t),
\end{equation}
where
\begin{flalign}
  &h_{2}(t)=\frac{1}{4}\alpha^{-1}\p_t\alpha,~m_{2}(t)=\sqrt{\frac{1}{m_{-}}}\alpha^{1/2},\label{eepfl:hmt}\\
  &p_{2}(t)=\frac{4\pi e^2}{\sqrt{m_{-}}\alpha^{1/2}}+\frac{2k^2\sqrt{m_{-}}}{\alpha^{3/2}}+\sqrt{\frac{1}{m_{-}}}\alpha^{1/2}.\label{eepfl:pmt}
\end{flalign}Then there exists positive constants $C_1$, $C^\sigma_1$, $C_2$, $C^\sigma_2$  that do not depend on $k$ and $\ze$ such that
\begin{equation}\label{eepeq:c1e}
  C_1\ca{E}_{-}(0) \leq \ca{E}_{-}(t) \leq C^\sigma_1\ca{E}_{-}(0),
\end{equation}and
\begin{equation}\label{eepeq:c2a}
  C_2\abs{B^{in}} \leq \abs{B(t)} \leq C^\sigma_2\abs{B^{in}}.
\end{equation}
\end{lemma}
\begin{proof}
Define
\begin{gather}
  \l_{-}=\sqrt{\frac{p_{2}}{m_{2}}}=\xkh{1+\frac{4\pi e^2}{\alpha}+\frac{2m_{-}k^2}{\alpha^2}}^{1/2},\label{eepga:lpm}\\
  \g_{-}=\sqrt{m_{2}p_{2}}=\xkh{\frac{4\pi e^2}{m_{-}}+\frac{2k^2}{\alpha}+\frac{\alpha}{m_{-}}}^{1/2}.\label{eepga:gmp}
\end{gather}By the definition of $\l_{-}$, it holds that
\begin{equation}\label{eepeq:lme}
  1 \leq (\l_{-})^2 \leq 1+4\pi e^2+2m_{-}.
\end{equation}
Moreover, we have
\begin{flalign}
  \frac{|h_{2}|}{\g_{-}}&=\frac{|\p_t \alpha|}{4\alpha}\xkh{\frac{4\pi e^2}{m_{-}}+\frac{2k^2}{\alpha}+\frac{\alpha}{m_{-}}}^{-1/2}\nonumber\\
  &=\frac{|\p_t \alpha|}{4\alpha}\xkh{\frac{m_{-}}{\alpha}}^{1/2}\xkh{1+\frac{4\pi e^2}{\alpha}+\frac{2m_{-}k^2}{\alpha^2}}^{-1/2}\nonumber\\
  &\leq\frac{|k|}{2\sqrt{\alpha}}\xkh{\frac{m_{-}}{\alpha}}^{1/2}\xkh{\frac{\alpha^2}{2m_{-}k^2}}^{1/2}=\frac{\sqrt{2}}{4},\label{eepfl:hg2}
\end{flalign}where we have used the fact that $|\p_t \alpha|\leq2|k|\sqrt{\alpha}$.

Setting
\begin{equation*}
  \wi{\ca{E}}_{-}(t)=\xkh{\l_{-}|B_1|^2}(t)+\xkh{\frac{1}{\l_{-}}|B_2|^2}(t).
\end{equation*}Then from (\ref{eepfl:hg2}) it follows that
\begin{equation}\label{eepeq:E22}
  \xkh{1-\frac{\sqrt{2}}{4}}\wi{\ca{E}}_{-}(t)\leq\ca{E}_{-}(t)\leq\xkh{1+\frac{\sqrt{2}}{4}}\wi{\ca{E}}_{-}(t).
\end{equation}Obviously, the coerciveness of $\ca{E}_{-}(t)$ is ensured by the above inequality and (\ref{eepeq:lme}).

According to (\ref{eepfl:hmt}), (\ref{eepfl:pmt}), (\ref{eepga:lpm}) and (\ref{eepga:gmp}), system (\ref{eepeq:ddt}) becomes
\begin{gather}
  \l_{-}\frac{d}{dt}B_1=-h_{2}\l_{-} B_1-\g_{-} B_2,\label{eepga:la1}\\
  \frac{1}{\l_{-}}\frac{d}{dt}B_2=\g_{-} B_1+\frac{h_{2}}{\l_{-}}B_2.\label{eepga:la2}
\end{gather}Multiplying equations (\ref{eepga:la1}) and (\ref{eepga:la2}) by $\bar{B_1}$ and $\bar{B_2}$, respectively, we obtain
\begin{equation}\label{eepeq:dl2}
  \frac{\l_{-}}{2}\frac{d}{dt}|B_1|^2+\frac{1}{2\l_{-}}\frac{d}{dt}|B_2|^2=-h_{2}\l_{-}|B_1|^2+\frac{h_{2}}{\l_{-}}|B_2|^2.
\end{equation}Note that
\begin{equation}\label{eepeq:hgd}
  \frac{h_{2}}{\g_{-}}\frac{d}{dt}\Re(B_1\bar{B_2})=h_{2}\l_{-}|B_1|^2-\frac{h_{2}}{\l_{-}}|B_2|^2.
\end{equation}Adding (\ref{eepeq:dl2}) to (\ref{eepeq:hgd}), we get
\begin{equation*}
  \l_{-}\frac{d}{dt}|B_1|^2+\frac{2h_{2}}{\g_{-}}\frac{d}{dt}\Re(B_1\bar{B_2})+\frac{1}{\l_{-}}\frac{d}{dt}|B_2|^2=0,
\end{equation*}which yields
\begin{equation*}
  \frac{d\ca{E}_{-}}{dt}=\l_{-}|B_1|^2\frac{d}{dt}\xkh{\log\l_{-}}+2\Re(B_1\bar{B_2})\frac{d}{dt}\xkh{\frac{h_{2}}{\g_{-}}}
  -\frac{1}{\l_{-}}|B_2|^2\frac{d}{dt}\xkh{\log\l_{-}}.
\end{equation*}
Using Young's inequality, from (\ref{eepeq:E22}) we deduce that
\begin{flalign}
  \frac{d\ca{E}_{-}}{dt}&\leq\xkh{\abs{\frac{d}{dt}\xkh{\frac{h_{2}}{\g_{-}}}}+\abs{\frac{d}{dt}\xkh{\log\l_{-}}}}\wi{\ca{E}}_{-}\nonumber\\
  &\leq\xkh{2+\frac{\sqrt{2}}{2}}\xkh{\abs{\frac{d}{dt}\xkh{\frac{h_{2}}{\g_{-}}}}+\abs{\frac{d}{dt}\xkh{\log\l_{-}}}}\ca{E}_{-}.\label{eepfl:dca}
\end{flalign}An argument similar to that of (\ref{eepfl:dca}) gives
\begin{equation}\label{eepfl:geq}
  \frac{d\ca{E}_{-}}{dt}\geq-\xkh{2+\frac{\sqrt{2}}{2}}\xkh{\abs{\frac{d}{dt}\xkh{\frac{h_{2}}{\g_{-}}}}+\abs{\frac{d}{dt}\xkh{\log\l_{-}}}}\ca{E}_{-}.
\end{equation}

For the purpose of applying Gr\"{o}nwall's inequality, we need to estimate the integral terms $\int^{+\infty}_0\abs{\frac{d}{dt}\xkh{\frac{h_{2}}{\g_{-}}}}d\vt$ and $\int^{+\infty}_0\abs{\frac{d}{dt}\xkh{\log\l_{-}}}d\vt$. First of all, we deal with the first one. By a direct calculation, we have
\begin{equation*}
  \frac{d}{dt}\xkh{\frac{h_{2}}{\g_{-}}}=\frac{\ca M}{8m_{-}(\g_{-})^3\alpha^3},
\end{equation*}where
\begin{equation*}
  \ca M=4m_{-}k^2(\g_{-})^2\alpha^2-2m_{-}(\g_{-})^2\alpha(\p_t\alpha)^2-\alpha^2(\p_t\alpha)^2+2m_{-}k^2(\p_t\alpha)^2
\end{equation*}is a polynomial in time of order 6. Denote by $t_i(i=1,2,...,6)$ the positive root for $\ca M$. We set $i_0=0$ provided that $\ca{M}(0) \leq 0$ and take $i_0=1$ for the other cases. Assume that $t_0=0$ and that $t_7=+\infty$. By virtue of (\ref{eepfl:hg2}), we conclude that
\begin{equation}\label{eepeq:ab7}
  \int^{+\infty}_0\abs{\frac{d}{dt}\xkh{\frac{h_{2}}{\g_{-}}}}d\vt=\sum^7_{i=1}(-1)^{i+i_0}\xkh{\frac{h_{2}}{\g_{-}}(t_i)-\frac{h_{2}}{\g_{-}}(t_{i-1})}
  \leq2\sum^7_{i=0}\frac{|h_{2}|}{\g_{-}}(t_i) \leq 4\sqrt{2}.
\end{equation}For the second integral term, we get from (\ref{eepeq:lme}) that
\begin{flalign}
  &\int^{+\infty}_0\abs{\frac{d}{dt}\xkh{\log\l_{-}}}d\vt
  =\frac{1}{2}\int^{\frac{\ze}{k}}_0\abs{\frac{d}{dt}\xkh{\log(\l_{-})^2}}d\vt
  +\frac{1}{2}\int^{+\infty}_{\frac{\ze}{k}}\abs{\frac{d}{dt}\xkh{\log(\l_{-})^2}}d\vt\nonumber\\
  &\qquad\leq\frac{1}{2}\zkh{\xkh{\log(\l_{-})^2}\xkh{\frac{\ze}{k}}-\xkh{\log(\l_{-})^2}\xkh{0}}
  +\frac{1}{2}\zkh{\xkh{\log(\l_{-})^2}\xkh{\frac{\ze}{k}}-\xkh{\log(\l_{-})^2}\xkh{+\infty}}\nonumber\\
  &\qquad\leq\log\xkh{1+4\pi e^2+2m_{-}}.\label{eepfl:me0}
\end{flalign}Here, we have used the fact that $\frac{1}{(\l_{-})^2}\frac{d (\l_{-})^2}{dt}=(\l_{-})^{-2}\alpha^{-3}\xkh{-4m_{-}k^2-4\pi e^2\alpha}\p_t\alpha$ changes sign only in $t=\ze/k$. Applying Gr\"{o}nwall's inequality to (\ref{eepfl:dca}) and (\ref{eepfl:geq}), respectively, from (\ref{eepeq:ab7}) and (\ref{eepfl:me0}) we obtain (\ref{eepeq:c1e}).

Next, it remains to prove (\ref{eepeq:c2a}). Owing to (\ref{eepeq:c1e}) and (\ref{eepeq:E22}), it can be deduced that
\begin{equation}\label{eepeq:cae}
  C'_1\wi{\ca{E}}_{-}(0) \leq \wi{\ca{E}}_{-}(t) \leq C^{\sigma'}_1\wi{\ca{E}}_{-}(0).
\end{equation}
By virtue of (\ref{eepeq:lme}), we have
\begin{equation}\label{eepeq:wic}
  \xkh{1+4\pi e^2+2m_{-}}^{-1}\wi{\ca{E}}_{-}(t) \leq |B(t)|^2
  \leq \xkh{1+4\pi e^2+2m_{-}}\wi{\ca{E}}_{-}(t).
\end{equation}
Combining (\ref{eepeq:cae}) with (\ref{eepeq:wic}) yields (\ref{eepeq:c2a}). This completes the proof.
\end{proof}

Based on Lemma \ref{eeple:aet}, we give the proof of Theorem \ref{eepth:uqp} as follows.
\begin{proof}[Proof of Theorem \ref{eepth:uqp}]
Thanks to Lemma \ref{eeple:aet}, an argument similar to that of (\ref{fl:ata}) yields
\begin{flalign}
  |B(t)|&=\abs{S_L(t,0)B^{in}+\int^t_0S_L(t,s)M_{-}(s)\ha{F^{in}_{-}}ds}\nonumber\\
  &\lesssim|B^{in}|+\abs{\ha{F^{in}_{-}}}\label{eepfl:ata}
\end{flalign}for any $t\geq0$. By virtue of the Helmholtz decomposition and the Plancherel's theorem, we deduce from (\ref{eepfl:ata}) that
\begin{flalign}
  \norm{\bb{P}[u_{-}]^x(t)}^2_{L^2}&+\norm{\phi(t)}^2_{L^2}=\norm{\xkh{\p_y\la^{-1}\omega_{-}}(t)}^2_{L^2}
  +\norm{4\pi e\xkh{-\la}^{-1}\eta_{-}(t)}^2_{L^2}\nonumber\\
  &=\sum_k\int\frac{(\ze-kt)^2}{\alpha^2}\abs{\ha{\Gamma_{-}}(t)}^2d\ze
  +\sum_k\int\frac{16\pi^2e^2}{\alpha^2}\abs{\ha{\Pi_{-}}(t)}^2d\ze\nonumber\\
  &\lesssim\sum_k\int\xkh{\frac{(\ze-kt)^2}{\alpha^2}\abs{\ha{F^{in}_{-}}}^2
  +\frac{m_{-}(\ze-kt)^2}{\alpha^{3/2}}\abs{B_1(t)}^2
  +\frac{m_{-}}{\alpha^{3/2}}\abs{B_1(t)}^2}d\ze\nonumber\\
  &\lesssim\sum_k\int\xkh{\alpha^{-1}\abs{\ha{F^{in}_{-}}}^2
  +\frac{m_{-}}{\alpha^{1/2}}\abs{B^{in}}^2
  +\frac{m_{-}}{\alpha^{1/2}}\abs{\ha{F^{in}_{-}}}^2}d\ze\nonumber\\
  &\lesssim\sum_k\int\xkh{\frac{\lan{\ze}^2}{\lan{t}^2\lan{k}^2}\abs{\ha{F^{in}_{-}}}^2
  +\frac{m_{-}\lan{\ze}}{\lan{t}\lan{k}}\abs{B^{in}}^2
  +\frac{m_{-}\lan{\ze}}{\lan{t}\lan{k}}\abs{\ha{F^{in}_{-}}}^2}d\ze,
\end{flalign}
in which we have used the fact that $\lan{t}\lesssim\lan{\ze/k-t}\lan{\ze/k}$. By the definition of the anisotropic Sobolev space, we obtain (\ref{eepfl:pux}) from the above inequality. Moreover, a similar argument as that of (\ref{eepfl:pux}) gives (\ref{eepfl:puy}).

Owing to the Helmholtz decomposition and the Plancherel's theorem, we conclude that
\begin{flalign}
  &\norm{\bb{Q}[u_{-}](t)}^2_{L^2}+\frac{1}{m_{-}}\norm{\eta_{-}(t)}^2_{L^2}
  =\norm{\xkh{\n\la^{-1}\psi_{-}}(t)}^2_{L^2}+\frac{1}{m_{-}}\norm{\eta_{-}(t)}^2_{L^2}\nonumber\\
  &\qquad=\norm{\xkh{\p_x\la^{-1}\psi_{-}}(t)}^2_{L^2}
  +\norm{\xkh{\p_y\la^{-1}\psi_{-}}(t)}^2_{L^2}+\frac{1}{m_{-}}\norm{\eta_{-}(t)}^2_{L^2}\nonumber\\
  &\qquad=\sum_k\int\xkh{\frac{1}{k^2+(\ze-kt)^2}\abs{\ha{\Psi_{-}}(t)}^2
  +\frac{1}{m_{-}}\abs{\ha{\Pi_{-}}(t)}^2}d\ze.\label{eepfl:ut2}
\end{flalign}
Note that the fact that $k^2+(\ze-kt)^2\lesssim\lan{t}^2\lan{k,\ze}^2$. From (\ref{eepfl:ata}) and (\ref{eepfl:ut2}), it holds that
\begin{flalign}
  &\norm{\bb{Q}[u_{-}](t)}^2_{L^2}+\frac{1}{m_{-}}\norm{\eta_{-}(t)}^2_{L^2}
  =\sum_k\int\xkh{\frac{1}{\alpha}\abs{\ha{\Psi_{-}}(t)}^2+\frac{1}{m_{-}}\abs{\ha{\Pi_{-}}(t)}^2}d\ze\nonumber\\
  &\qquad=\sum_k\int\alpha^\frac{1}{2}\xkh{\abs{\alpha^{-\frac{3}{4}}\ha{\Psi_{-}}(t)}^2
  +\abs{\sqrt{\frac{1}{m_{-}}}\alpha^{-\frac{1}{4}}\ha{\Pi_{-}}(t)}^2}d\ze
  =\sum_k\int\alpha^\frac{1}{2}\abs{B(t)}^2d\ze\nonumber\\
  &\qquad\lesssim \lan{t}\xkh{\norm{\sqrt{\frac{1}{m_{-}}}\eta^{in}_{-}}^2_{L^2}
  +\norm{\psi^{in}_{-}}^2_{H^{-1}}+\norm{\eta^{in}_{-}+\omega^{in}_{-}}^2_{H^1}},
\end{flalign}which yields estimate (\ref{eepfl:qut}).
\end{proof}

The proof of Theorem \ref{eeth:que} follows from a similar argument as that of Theorem \ref{th:que}. Hence we omit its proof here.

\paragraph{Acknowledgements.\MC\MC\MC}\small
\addcontentsline{toc}{section}{Acknowledgements}
The work of X. Pu was supported in part by the National Natural Science Foundation of China (11871172) and the Science and Technology Projects in Guangzhou (202201020132). The work of D. Bian was supported in part by the National Natural Science Foundation of China (12271032).

\paragraph{}

\end{document}